\documentclass[11pt,leqno]{article}
\usepackage{amsmath,amssymb,mathrsfs,xy,amsthm,bm}
\input{xy}
\xyoption{all}

\usepackage[hypertex]{hyperref}

\makeatletter

 \def\@seccntformat#1{\csname the#1\endcsname.\hspace{2ex}}

 \newcommand{\nsubsection}%
  {\@startsection{subsection}%
  {2}%
  {\z@}%
  {-3.5ex plus -1ex minus -.2ex}%
  {-0ex}%
  {\reset@font\normalsize\bfseries}}%

 \newcommand{\nnsubsubsection}%
  {\@startsection{subsubsection}%
  {3}%
  {\z@}%
  {-3.5ex plus -1ex minus -.2ex}%
  {-2ex}%
  {\reset@font\normalsize\bfseries}}%

  \renewcommand{\subsubsection}%
  {\@startsection{subsubsection}%
  {3}%
  {\z@}%
  {-3.5ex plus -1ex minus -.2ex}%
  {0ex}
  {\reset@font\normalsize\bfseries}}%

  \renewcommand{\subsection}%
  {\@startsection{subsection}%
  {2}%
  {\z@}%
  {-3.5ex plus -1ex minus -.2ex}%
  {0ex}
  {\reset@font\normalsize\bfseries}}%

 \newcommand{\nnsubsection}%
  {\@startsection{subsection}%
  {2}%
  {\z@}%
  {-3ex}%
  {1ex}%
  {\reset@font\normalsize\bfseries}}%

 \newcommand{\usubsection}%
  {\@startsection{subsection}%
  {2}%
  {\z@}%
  {-3.5ex plus -1ex minus -.2ex}%
  {0.5ex}
  {\reset@font\normalsize\bfseries}}%

 \@addtoreset{subsection}{section}
 \@addtoreset{equation}{subsection}

 \newcommand{\nsection}{\@startsection{section}{1}{\z@}%
     {-5ex}
     {1ex}
     {\reset@font\center\large\sc}}

 \renewenvironment{thebibliography}[1]
 {\nsection*{\refname\@mkboth{\refname}{\refname}}%
   \list{\@biblabel{\@arabic\c@enumiv}}%
        {\settowidth
	\labelwidth{\@biblabel{#1}}%
         \leftmargin
	 \labelwidth
         \advance
	 \leftmargin
	 \labelsep
         \@openbib@code
         \usecounter{enumiv}%
         \let\p@enumiv\@empty
	 \parskip=0pt
	 \itemsep=1pt
	 \parsep=1pt
	 \itemindent=\z@
         \renewcommand\theenumiv{\@arabic\c@enumiv}}%
   \sloppy
   \clubpenalty4000
   \@clubpenalty\clubpenalty
   \widowpenalty4000%
   \footnotesize
   \sfcode`\.\@m}
  {\def\@noitemerr
    {\@latex@warning{Empty `thebibliography' environment}}%
   \endlist}
\makeatother

\newtheoremstyle{thm}
 {1em}
 {3pt}
 {\itshape}
 {}
 {\bf}
 {. ---}
 {0.5em}
 {}
\newtheoremstyle{dfn}
 {1em}
 {3pt}
 {}
 {}
 {\bf}
 {. {---}}
 {0.5em}
 {}

\swapnumbers
\theoremstyle{thm}
\newtheorem{thm}[subsection]{Theorem}
\newtheorem{lem}[subsection]{Lemma}
\newtheorem*{lem*}{Lemma}
\newtheorem{cor}[subsection]{Corollary}
\newtheorem*{cor*}{Corollary}
\newtheorem{prop}[subsection]{Proposition}
\newtheorem*{prop*}{Proposition}

\newtheorem*{thm*}{Theorem}

\theoremstyle{dfn}

\newtheorem*{dfn*}{Definition}

\newtheorem*{ex*}{Example}
\newtheorem{rem}[subsection]{Remark}
\newtheorem*{rem*}{Remark}

\usepackage{color}
\newenvironment{meta}{
\noindent \color{red}
\sffamily[}{\upshape]}

\newcommand{\riso}{ \overset{\sim}{\longrightarrow}\, }

\newcommand{\coker}{\mathrm{coker}\,}

\newcommand{\FF}{{\mathcal{F}}}
\newcommand{\B}{{\mathcal{B}}}

\newcommand{\E}{{\mathcal{E}}}
\newcommand{\G}{{\mathcal{G}}}
\renewcommand{\H}{{\mathcal{H}}}

\newcommand{\D}{{\mathcal{D}}}
\newcommand{\I}{{\mathcal{I}}}
\newcommand{\J}{{\mathcal{J}}}
\newcommand{\PP}{{\mathcal{P}}}

\renewcommand{\O}{{\mathcal{O}}}

\newcommand{\V}{\mathcal{V}}
\renewcommand{\S}{\mathcal{S}}

\newcommand{\ZZ}{\mathcal{Z}}
\newcommand{\X}{\mathcal{X}}

\newcommand{\U}{\mathcal{U}}

\newcommand{\A}{\mathbb{A}}

\newcommand{\DD}{\mathbb{D}}

\newcommand{\R}{\mathbb{R}}
\newcommand{\Q}{\mathbb{Q}}
\newcommand{\Z}{\mathbb{Z}}

\newcommand{\hdag}{  \phantom{}{^{\dag} }    }

\newcommand{\mr}[1]{\mathrm{#1}}
\newcommand{\ms}[1]{\mathscr{#1}}
\newcommand{\mc}[1]{\mathcal{#1}}
\newcommand{\mb}[1]{\mathbb{#1}}

\newcommand{\DdagQ}[1]{\mc{D}^\dag_{{#1},\mathbb{Q}}}
\newcommand{\Gammadag}{\underline{\Gamma}^\dag}
\newcommand{\fsch}[1]{{\mc{#1}}}

\renewcommand{\H}{\ms{H}}

\newcommand{\indlim}{\mathop{\underrightarrow{\mathrm{lim}}}}
\newcommand{\shom}{\mc{H}om}

\begin{document}
\title{On Beilinson's equivalence for $p$-adic cohomology}
\author{Tomoyuki Abe and Daniel Caro}
\date{}
\maketitle

\begin{abstract}
 In this short paper, we construct a unipotent nearby cycle functor and
 show a $p$-adic analogue of Beilinson's equivalence comparing two
 derived categories: the derived category of holonomic modules and
 derived category of modules whose cohomologies are holonomic.
\end{abstract}

\section*{Introduction}
In the theory of $p$-adic cohomology, lack of nearby cycle functor has
been a big technical obstruction for proving important results.
For example, \cite{AM}, \cite{Caprop} are few of such examples.
In this short paper, we establish the theory of unipotent nearby cycle
functor, and as an application, we prove a $p$-adic analogue of
Beilinson's equivalence: for a smooth variety $X$ over $\mb{C}$, we have
an equivalence of categories (see \cite{Be})
\begin{equation*}
 D^{\mr{b}}(\mr{Hol}(X))\xrightarrow{\sim}
  D^{\mr{b}}_{\mr{hol}}(X).
\end{equation*}

For the construction of the unipotent nearby cycle functor, we follow
the idea of \cite{Beglue}.
The original construction of Beilinson's unipotent nearby
cycles in the context of algebraic $\D$-modules is based on a key lemma
whose proof is a consequence of the existence of $b$-functions.
However, in our $p$-adic context, the definition of $b$-functions is
problematic.
To remedy this, we can use successfully another powerful tool, namely,
Kedlaya's semistable reduction theorem, applied to overconvergent
isocrystals with Frobenius structure.

Now, even though the proof of Beilinson's equivalence is written in a
way that it can be adopted for many cohomology theories, we still need
to figure out what the suitable definition of ``holonomic modules'' are
in the $p$-adic context.
A naive answer might be to consider overholonomic complexes (without
Frobenius structure) introduced by the
second author. However, we do not know if this category is closed under
taking tensor products when modules do not admit Frobenius
structure. Thus, the category does not seem appropriate for the
equivalence because Beilinson's original proof uses the stability under
Grothendieck six operations. Moreover, the full subcategory of
overholonomic modules whose objects are endowed with some Frobenius
structure is not thick. To resolve these issues, in this paper, we
construct some kind of smallest triangulated subcategory of the category
of overholonomic complexes which contains modules with Frobenius
structure. Its construction allows us to come down by ``devissage'' to
the case of modules with Frobenius structure.

Finally, we point out that techniques developed in this paper are
crucial tools to construct the theory of arithmetic $\D$-modules for
general schemes in \cite{Apcc}, and we also expect more applications:
unification of the rigid cohomology theory into arithmetic $\D$-modules
(cf.\ \cite[1.3.11]{Apcc}), $p$-adic analogue of Fujiwara's trace
formula, {\it etc.}.

The first section is devoted to construct the good triangulated
category, and the unipotent nearby cycle functors is treated in the
second section.

\subsection*{Acknowledgment}\mbox{}\\
The first author (T.A.) was supported by Grant-in-Aid for Young
Scientists (B) 25800004.
The second author (D.C.) thanks Thomas Bitoun for his interest
concerning this $p$-adic analogue.
\bigskip

In this paper, we fix a complete discrete valuation ring $R$ of mixed characteristic. Its
residue field is denoted by $k$, and we assume it to be perfect and of characteristic $p$. 
We suppose that there exists a lifting $\sigma\colon R\xrightarrow{\sim}R$ of the $s$-th
Frobenius automorphism of $k$. We put $q:=p^s$, $K:=\mr{Frac}(R)$.
If there is no ambiguity with $K$,
we sometimes omit ``$/K$'' in the notation of some categories.

\section{Overholonomic $\DdagQ{\fsch{X}}$-modules}

\subsection{On the stability under base change. }
Let $\fsch{P}$ be a smooth formal scheme over $R$.
Let $\E$ be an overholonomic $\DdagQ{\fsch{P}/\S}$-module (resp.\ an
overholonomic complex of $\DdagQ{\fsch{P}/\S}$-modules).
Recall that $\E$ is said to be {\em overholonomic after any base change}
if for any morphism $k \to k'$ of perfect fields, 
putting 
$R' := R \otimes _{W (k)} W (k')$, 
$\S := \mathrm{Spf} (R)$, $\S ':=
\mathrm{Spf} (R')$, 
$f\colon\fsch{P} ':= \fsch{P}\times _{\S} \S ' \to \fsch{P}$
the canonical morphism, then the object
$f ^* (\E):= \D ^\dag _{\fsch{P}'/\S', \Q} \otimes _{f ^{-1} \D ^\dag
_{\fsch{P}/\S,\Q}} f ^{-1} \E$ remains to be an overholonomic
$\DdagQ{\fsch{P'}/\S'}$-module (resp.\
complex of $\DdagQ{\fsch{P}'/\S'}$-modules).

We remark that the base change functor $f ^{*}$ is exact, commutes with
push-forwards, pull-backs, dual functors, local cohomological functors and
preserves the coherence and the holonomicity (use Virrion's
characterization of the holonomicity of \cite[III.4]{virrion}). For instance,
if $Y$ is a subvariety of the special fiber of $\fsch{P}$ and $Y':= f
^{-1} (Y)$, for any overholonomic complex $\E$ of
$\DdagQ{\fsch{P}/\S}$-modules, we get the isomorphism of coherent
complexes $\R \underline{\Gamma} ^{\dag} _{Y '} (f ^{*} \E) \riso f ^{*}
\R \underline{\Gamma} ^{\dag} _{Y} (\E)$ of
$\DdagQ{\fsch{P}'/\S'}$-modules.

\begin{lem}
 \label{lemstabbchg}
 Let $\fsch{P}$ be a proper smooth formal scheme over $R$, and $\E$ be an
 object of $F \text{-} D ^\mathrm{b} _{\mathrm{ovhol}}
 (\DdagQ{\fsch{P}})$, {\it i.e.}\ an overholonomic complex of
 $\DdagQ{\fsch{P}}$-modules endowed with Frobenius structure.
 Then $\E$ is overholonomic after any base change. 
\end{lem}

\begin{proof}
Let   $k \to k'$ be a morphism of perfect fields, 
$R' := R \otimes _{W (k)} W (k')$, 
$\S := \mathrm{Spf} (R)$, $\S ':=
 \mathrm{Spf} (R')$, $f\colon\fsch{P} ':= \fsch{P} \times _{\S} \S
 '\to\fsch{P}$ be the canonical morphism.
 We have to prove that $f ^{*} (\E)$ is overholonomic. By devissage, we
 can suppose that there exists a quasi-projective subvariety $Y$ of the
 special fiber of $\fsch{P}$ such that $\E \in F \text{-} D ^\mathrm{b}
 _{\mathrm{ovhol}} (Y, \fsch{P})$, {\it i.e.}\ by definition of this
 category such that $\R \underline{\Gamma} ^{\dag} _{Y} (\E) \riso \E$.
 There exists an immersion of the form $Y \hookrightarrow \fsch{Q}$,
 where $\fsch{Q}$ is a projective formal scheme over $R$.
 We get an immersion $Y \hookrightarrow \fsch{P} \times \fsch{Q}$ and
 two projections
 $p _{1}\colon\fsch{P} \times \fsch{Q}\to  \fsch{P}$,
 $p _{2}\colon\fsch{P} \times \fsch{Q}\to  \fsch{Q}$.
 We recall that the categories
 $F \text{-} D ^\mathrm{b} _{\mathrm{ovhol}} (Y, \fsch{P})$ and
 $F \text{-} D ^\mathrm{b} _{\mathrm{ovhol}} (Y, \fsch{Q})$
 are canonically equivalent.
 Let $\FF$ be the object of $F \text{-} D ^\mathrm{b} _{\mathrm{ovhol}} (Y,
 \fsch{Q})$ corresponding to $\E$, {\it i.e.}\ $\E \riso p _{1+} \R
 \underline{\Gamma} ^{\dag} _{Y} p ^{!} _2 (\FF)$.
 Let $Y '$, $ \fsch{Q} '$, $p ' _1$, $p ' _2$ be the base change of $Y
 $, $ \fsch{Q}$, $p  _1$, $p _2$ by $f$. The complex $f ^{*} (\FF)$
 is endowed with a Frobenius structure by using \cite[2.1.6]{Be2} and is
 holonomic because $f ^{*}$ preserves the holonomicity.
 Therefore $f
 ^{*} (\FF)$ is overholonomic by \cite{caro-stab-holo} since $\fsch{Q}'$
 is projective. Since $f ^{*} (\E) \riso p ' _{1+}\R
 \underline{\Gamma}^{\dag} _{Y'}p ^{\prime !} _2 (f
 ^{*}\FF)$, the stability of overholonomicity implies that
 $f ^{*} (\E)$ is also overholonomic.
\end{proof}

\begin{lem}
 \label{lemm1.1}
 Let $\fsch{P}$ be a smooth formal scheme over $R$.
 We denote by $\mr{Ovhol}(\fsch{P})$ the subcategory of the category
 $\mr{Mod}(\DdagQ{\fsch{P}})$ of $\DdagQ{\fsch{P}}$-modules consisting of
 overholonomic $\DdagQ{\fsch{P}}$-modules after any base change.
 The category $\mr{Ovhol}(\fsch{P})$ is a thick abelian subcategory of
 $\mr{Mod}(\DdagQ{\fsch{P}})$.
\end{lem}
\begin{proof}
 To check this, we need to show that kernel and cokernel are
 overholonomic. Let $\E\rightarrow\FF$ be a homomorphism of
 overholonomic modules. Then these are holonomic by
 \cite[4.3]{Cahol}. Thus the kernel and cokernel are holonomic by [{\it
 ibid.}, 2.14]. Since the functor $\mb{D}$ is exact on the category of
 holonomic modules, we get the overholonomicity of kernel and cokernel.
 The thickness can be seen easily.
\end{proof}

\subsection{}
A variety ({\it i.e.}\ a
reduced scheme of finite type over $k$) $X$ is said to be {\em
realizable} if there exists a smooth proper formal scheme $\fsch{P}$
over $R$ such that $X$ can be embedded into $\fsch{P}$. Since the
cohomology theory does not change if we take the associated reduced
scheme, in the following, we assume that schemes are always reduced.
For any realizable variety $X$, choose $X\hookrightarrow\fsch{P}$ an
immersion with $\fsch{P}$ a smooth proper formal scheme over $R$.
Then by \cite[4.16]{Caovhol}, the category of overholonomic
$\DdagQ{\fsch{P}}$-complexes $\E$ after any base change which is
supported on $\overline{X}$ and $\mb{R}\Gammadag_{\overline{X}\setminus
X} (\E)=0$ does not depend on the choice of $\fsch{P}$. This category
is denoted by $D^{\mr{b}}_{\mr{ovhol}}(X/K)$ (in this paper, from now
on, we only work with overholonomic complexes after any base change, so
we keep the notation concerning overholonomic complexes in order to
avoid getting too heavy notation).

Let $X$ be a realizable variety. From \cite{AC}, we define a
t-structure on $D^{\mr{b}}_{\mr{ovhol}}(X/K)$, and its heart is denoted
by $\mr{Ovhol}(X/K)$.

\begin{lem}
 Let $X$ be a realizable variety. Then for any
 overholonomic after any base change module $\E\in\mr{Ovhol}(X/K)$, any
 ascending or descending chain of overholonomic submodules of $\E$ is
 stationary.
\end{lem}
\begin{proof}
 By base change, we can suppose that $k$ is uncountable.
 We prove the claim using the induction on
 the dimension of the support. Let $\E\in\mr{Ovhol}(X)$. From
 \cite[3.7]{Cahol} \footnote{In the statement of \cite[3.7]{Cahol}, 
 we need to add that $k$ is uncountable or that
 the property to have finite fibers is stable under base change.} 
there exists an open dense subscheme $U$ of $X$ such
 that $X\setminus U$ is a divisor and $\G :=  \E|_U\in\mr{Isoc} ^{\dag
 \dag}(U)$. By induction hypothesis, we reduce to check that $j_+(\G)$ satisfies 
 the ascending (resp. descending) chain condition.
 Take an irreducible submodule
 $\G'\subset\G$ in $\mr{Ovhol}(U)$. From \cite[1.4.7]{AC},  since $\G'$
 is irreducible then so is  $j_{!+}(\G')$. Thus by induction hypothesis,
 $j_+(\G')$ satisfies
 the ascending (resp. descending) chain condition.
 Since $j _+$ is exact, if $\G$ is not irreducible then 
 we conclude by using a second induction on the generic rank (on the
 rigid analytic spaces) of $\G$.
\end{proof}

\begin{rem*}
 For a smooth formal scheme $\fsch{P}$ (which may not be proper), we may
 also show that any overholonomic module on $\fsch{P}$ satisfies the
 ascending and descending chain conditions. The proof is similar.
\end{rem*}

\begin{cor}
 \label{constiFrobend}
 Let $\E\in\mr{Ovhol}(X)$, and assume that $\E$ can be
 endowed with a $s'$-th Frobenius structure for an integer $s'$ which is
 a multiple of $s$. Then any constituents of $\E$ in $\mr{Ovhol}(X)$ can
 be endowed with a $s''$-th Frobenius structure for some $s''$  a
 multiple of $s'$.
\end{cor}
\begin{proof}
 The verification is similar to \cite[6.0-15]{CM}.
\end{proof}

\subsection{}
Let $X$ be a realizable variety.
Let $\mr{Hol}_F(X)'$ be the subcategory of $\mr{Ovhol}(X)$ whose objects 
can be endowed with $s'$-th Frobenius structure for some integer $s'$
which is a multiple of $s$, and let $\mr{Hol}_F(X)$ be the thick
abelian subcategory generated by $\mr{Hol}_F(X)'$ in
$\mr{Ovhol}(X)$. We denote by
$D^{\mr{b}}_{\mr{hol},F}(X)$ the triangulated full subcategory of
$D^{\mr{b}}_{\mr{ovhol}}(X)$ such that the cohomologies are in
$\mr{Hol}_F(X)$. By Lemma \ref{lemm1.1} and Corollary
\ref{constiFrobend}, we have:

\begin{cor*}
 Any object of $\mr{Hol}_F(X)$ can be written as extensions of
 modules in $\mr{Hol}_F(X)'$.
\end{cor*}

This corollary has the following consequences:

\begin{thm}
 \label{sixfuncforma}
 Let $f\colon X\rightarrow Y$ be a morphism between realizable
 varieties.

 \begin{enumerate}
  \item If $f$ is proper, $f_+$ induces the functor
	$D^{\mr{b}}_{\mr{hol},F}(X)\rightarrow
	D^{\mr{b}}_{\mr{hol},F}(Y)$.

  \item The functor $f^!$ induces
	$D^{\mr{b}}_{\mr{hol},F}(Y)\rightarrow
	D^{\mr{b}}_{\mr{hol},F}(X)$.

  \item The functor $\mb{D}$ induces the functor
	$D^{\mr{b}}_{\mr{hol},F}(X)^{\circ}\rightarrow
	D^{\mr{b}}_{\mr{hol},F}(X)$ such that
	$\mb{D}\circ\mb{D}\cong\mr{id}$.

  \item\label{tensprodpres}
       The functor $\widetilde{\otimes}$ (cf.\ \cite[1.1.6
       (ii)]{AC}) induces $D^{\mr{b}}_{\mr{hol},F}(X)\times
       D^{\mr{b}}_{\mr{hol},F}(X) \rightarrow
       D^{\mr{b}}_{\mr{hol},F}(X)$.
 \end{enumerate}
 Moreover, these functors satisfy the properties listed in
 \cite[1.3.14]{AC}.
\end{thm}
We recall that when we take an embedding $X\hookrightarrow\PP$ into a
proper smooth formal scheme $\PP$, then $\widetilde{\otimes}$ can be
written as $(-)\otimes^\dag_{\O_{\PP}}(-)[-\dim(\PP)]$, where
$\otimes^\dag_{\O_{\PP}}$ is the usual weakly completed tensor product.
In the following, we introduce the functor
$\otimes\colon D^{\mr{b}}_{\mr{hol},F}(X)\times
D^{\mr{b}}_{\mr{hol},F}(X)\rightarrow D^{\mr{b}}_{\mr{hol},F}(X)$ to be
$\DD\bigl(\DD(-)\widetilde{\otimes}\DD(-)\bigr)$ as in
\cite[1.1.6 (iii)]{AC}.

\begin{rem*}
 Even if we replace $D^{\mr{b}}_{\mr{hol},F}$ by
 $D^{\mr{b}}_{\mr{ovhol}}$, the theorem holds except for
 \ref{tensprodpres}, which has been checked by the second author.
\end{rem*}

\subsection{}
Using this category, we can state our main theorem as follows:

\begin{thm*}
 Let $X$ be a realizable variety. Then the canonical functor
 \begin{equation*}
  D^{\mr{b}}(\mr{Hol}_F(X/K))\rightarrow
   D^{\mr{b}}_{\mr{hol},F}(X/K)
 \end{equation*}
 is an equivalence of categories.
\end{thm*}
\begin{proof}
 With the aid of the next section, the proof of \cite{Be} can be adapted
 without any difficulties, so we only sketch the outline.
 We put $M(X):=\mr{Hol}_F(X/K)$ and
 $D(X):=D^{\mr{b}}_{\mr{hol},F}(X/K)$.
 First task is to define the functor $\mr{real}_X\colon
 D^{\mr{b}}(M(X))\rightarrow D(X)$ which induces an identity on the
 hearts of t-structures. This can be defined using the abstract
 non-sense presented in the appendix of \cite{Be}.

 Now, for a generic point $\eta\in X$,
 we put $D(\eta):=2\text{-}\indlim_{\eta\in U}D(U)$, and
 $M(\eta):=2\text{-}\indlim_{\eta\in U}M(U)$.
 We prove that the functor
 \begin{equation}
  \label{genericequiv}
  \mr{real}_{\eta}\colon D^{\mr{b}}(M(\eta))\rightarrow D(\eta)
 \end{equation} 
 is an equivalence. For the
 proof, we need six functors formalism as we constructed in Theorem
 \ref{sixfuncforma}, and we can copy \cite[2.1]{Be}: Let $\eta\in
 U\subset X$ be an open subscheme, and $M_U$, $N_U$ are in $M(U)$. It
 suffices to show the existence of an open subscheme $\eta\in
 V\xrightarrow{j}U$, $O_V\in M(V)$ and $N_V:=j^+N_U\hookrightarrow O_V$
 such that the induced homomorphism
 $\mr{Ext}^i_{D(U)}(M_U,N_U)\rightarrow\mr{Ext}^i_{D(V)}(M_V,O_V)$ is zero.
 We use the induction on the dimension of $X$, and assume that the
 equivalence of (\ref{genericequiv}) holds for any $Y$ of dimension less
 than $X$. By shrinking $U$, we may assume that $U$ is smooth, $M_U$ and
 $N_U$ are contained in $\mr{Isoc}^{\dag\dag}(U)$, there exists a
 smooth affine morphism $\P\colon U\rightarrow Z$ with $1$-dimensional
 fibers such that $Z$ is smooth and $L^q:=\H^q\P_+\shom(M_U,N_U)$
 are in $\mr{Isoc}^{\dag\dag}(U)$ for any $q$.
 We refer to \cite[A.5]{AC} for the relation between $\shom$ and
 $\mr{Hom}_{D(X)}$. For an open subscheme $Y\subset Z$, let
 $U_Y:=\P^{-1}(Y)$ and $\P_Y\colon U_Y\rightarrow Y$ is the one induced
 by $\P$. Since $\P$ is
 assumed to be affine and the dimension of each fiber is
 $1$-dimensional, we see that the Leray spectral sequence
 $E^{p,q}_2=\H^pp_{Z+}(L^q)\Rightarrow \mr{Ext}_{D(U)}^{p+q}(M_U,N_U)$
 degenerates at $E_3$, where $p_Z$ denotes the structural morphism of
 $Z$. Using this degeneration, Beilinson splits the
 construction problem of $O_V$ into two: one is an open subscheme
 $Y\subset Z$ and $N_{U_Y}\hookrightarrow P_{U_Y}$
 such that $\H^1\P_{Y+}\shom(M_{U_Y},N_{U_Y})\rightarrow
 \H^1\P_{Y+}\shom(M_{U_Y},P_{U_Y})$ is zero, and the other is an open
 subscheme $Y'\subset Z$ and $N_{U_{Y'}}\hookrightarrow Q_{U_{Y'}}$ such
 that $\H^p(Z,\H^0\P_+\shom(M_U,N_U))\rightarrow
 \H^p(Y',\H^0\P_{Y'+}(\shom(M_{U_{Y'}},Q_{U_{Y'}}))$.
 He constructs $P_{U_Y}$ and $Q_{U_{Y'}}$ by using the induction
 hypothesis on $Z$ and the cohomological functor formalism.
 Even thought the construction is technical, the argument is very
 general which can be copied just by using the existing formalism in the
 theory of arithmetic $\ms{D}$-modules, and we refer the further details
 to \cite[2.1]{Be}.
 \medskip

 Using the proven generic case, let us complete the proof. We use the
 induction on the dimension of $X$. Since $X$ is separated, any open
 immersion $j\colon U\rightarrow X$ is affine, and in particular, $j_+$
 sends $M(U)$ to $M(X)$ by \cite[1.3.13]{AC}.
 Thus, by standard argument, the claim is
 Zariski local, and we may assume $X$ to be affine. It suffices to show
 that for $M$, $N$ in $M(X)$, the homomorphism
 $\mr{Ext}^i_{M(X)}(M,N)\rightarrow\mr{Ext}^i_{D(X)}(M,N)$, where
 $\mr{Ext}^i$ denotes the Yoneda's Ext functor, is an isomorphism. Using
 the equivalence of (\ref{genericequiv}) and the formal
 properties of cohomological functors, it is a standard devissage
 argument to reduce to the case where the supports of
 $M$ and $N$ are dimension less than that of $X$ (cf.\
 \cite[2.2.2--2.2.4]{Be}).
 Take a morphism $f\colon X\rightarrow\mb{A}^1$ such that $Y:=f^{-1}(0)$
 contains the support of $M$ and $N$. By using the induction hypothesis,
 we have
 \begin{equation*}
  \mr{Ext}_{D(X)}^i(M,N)\xrightarrow[i^!]{\sim}
   \mr{Ext}^i_{D(Y)}(M,N)\cong\mr{Ext}^i_{M(Y)}(M,N),
 \end{equation*}
 where the inverse of the first isomorphism is $i_+$. It remains to
 show that the canonical homomorphism
 \begin{equation*}
  I\colon\mr{Ext}^i_{M(Y)}(M,N)\rightarrow
   \mr{Ext}^i_{M(X)}(M,N),
 \end{equation*}
 is a bijection for any $i$. For this we need the existence of the
 functors $\Phi_f$ and $\Xi_f$. These functors are defined and basic
 properties are shown in the next section (cf.\ Proposition
 \ref{mainlemthmcons}). In fact, the inverse of $I$ can be constructed
 as
 \begin{equation*}
  \Phi_{f*}\colon
  \mr{Ext}^i_{M(X)}(M,N)\xrightarrow{\Phi_f}
   \mr{Ext}^i_{M(Y)}(\Phi_f(M),\Phi_f(N))\cong
   \mr{Ext}^i_{M(Y)}(M,N)
 \end{equation*}
 where we used the exactness of $\Phi_f$ in the first homomorphism, and
 the isomorphism holds since $M$ and $N$ are supported on $Y$. Since
 $\Phi_{f*}\circ I=\mr{id}$, it remains to show that $I\circ
 \Phi_{f*}=\mr{id}$. To check this, for an
 extension class $0\rightarrow N\rightarrow
 C^1\rightarrow\dots\rightarrow C^i\rightarrow M\rightarrow0$ in $M(X)$,
 we need to show that the class $0\rightarrow
 N\rightarrow\Phi_f(C^1)\rightarrow\dots\rightarrow\Phi_f(C^i)\rightarrow
 N\rightarrow0$ is in the same class. For this, Beilinson constructs an
 ingenious sequence of homomorphisms connecting the two using $\Xi_f$.
 See \cite[2.2.1]{Be}.
\end{proof}

\begin{rem*}
 This theorem is a generalization of \cite[A.4]{AC}.
\end{rem*}

\section{Unipotent nearby cycle functor}

\subsection{}
Let $\Pi := \{(a,b) \in \Z ^2\,;\,  a \leq b\}$ be the partially ordered
set such that $(a,b) \leq (a',b') \Leftrightarrow a \geq a',\, b \geq
b'$. For an abelian category $\mathfrak{A}$, we denoted by $\mathfrak{A}
^{\Pi}$ the category of $\Pi$-shaped diagrams in $\mathfrak{A}$,
{\it i.e.}\ the category whose objects are $\E ^{\bullet,\bullet}= (\E
^{a,b} , \alpha ^{(a, b),(a',b')})$, where $(a, b), (a',b') $ runs
through elements of $\Pi$ so that $(a', b')\leq (a, b)$, $\E
^{a,b}$ belong to $\mathfrak{A}$, and $\alpha ^{(a, b),(a',b')}\colon 
\E ^{a',b'} \to \E ^{a,b}$ are morphisms of $\mathfrak{A}$, transitive with respect to the composition.
We denoted by $\mathfrak{A} ^{\Pi} _{\mathrm{a}}$ the full subcategory
of $\mathfrak{A} ^{\Pi}$ of objects $\E ^{\bullet,\bullet}= (\E ^{a,b} ,
\alpha ^{(a, b),(a',b')})$ such that, for any $a \leq b \leq c$, the
sequence $0 \to \E ^{b,c} \to \E ^{a,c} \to \E ^{a ,b}\to 0$ is
exact. These objects are called {\em admissible}. Since this subcategory
is closed under extension, this is an exact category so that the
canonical functor $\mathfrak{A} ^{\Pi} _{\mathrm{a}}\to \mathfrak{A}
^{\Pi} $ is exact.

Let $M$ (resp. $M ^{f}$) be the set of order-preserving maps $\phi\colon\Z \to \Z$ such
that $\lim_{i \to \pm \infty } \phi ( i) =  \pm \infty $ (resp. and also
$id +N \geq \phi  \geq id -N $ for some integer $N$ large enough).
For any $\phi \in M$, we put $\widetilde{\phi} (\E ^{\bullet, \bullet}) := (\E
^{\phi (a) , \phi (b)}) _{(a,b) \in \Pi}$.
Let $S$ (resp. $S ^f$) be the set of the canonical morphisms of the form
$\widetilde{\phi} (\E ^{\bullet, \bullet}) 
\to 
\widetilde{\psi} (\E ^{\bullet, \bullet})$,
where 
$\phi, \psi \in M$
(resp. $\phi, \psi \in M ^{f}$)
satisfy $\phi \geq \psi$ 
and 
$\E ^{\bullet, \bullet}
\in \mathfrak{A}
^{\Pi}$.
We denote by $S _{\mathrm{a}}$ (resp. $S _{\mathrm{a}} ^{f}$)
the elements of $S$ (resp. $S ^{f}$)
which are morphisms of $\mathfrak{A} ^{\Pi} _{\mathrm{a}}$ as well.

\begin{enumerate}
\item Following \cite[Appendix]{Beglue}, we put $\underleftrightarrow{\lim} \,
\mathfrak{A}:= S _{\mr{a}} ^{-1}\mathfrak{A} ^{\Pi} _{\mr{a}}$ and
$\underleftrightarrow{\lim} ^{\mathrm{ab}}\mathfrak{A}:= S ^{-1}
\mathfrak{A} ^{\Pi}$.
For any $\E _{\mathrm{a}}^{\bullet, \bullet}, \FF _{\mathrm{a}}
^{\bullet, \bullet}  \in \underleftrightarrow{\lim} \, \mathfrak{A}$ and
for any $\E ^{\bullet, \bullet}, \FF ^{\bullet, \bullet} \in
\underleftrightarrow{\lim} ^{\mathrm{ab}} \, \mathfrak{A}$ we have the
equalities
\begin{gather}
 \notag
 \mathrm{Hom} _{\underleftrightarrow{\lim} \, \mathfrak{A}}
 (\E _{\mathrm{a}}^{\bullet, \bullet}, \FF _{\mathrm{a}} ^{\bullet,
 \bullet} )=
 \underset{\phi \in M}{\underrightarrow{\lim}}\;
 \mathrm{Hom} _{\mathfrak{A} _{\mathrm{a}} ^{\Pi} }
 (\widetilde{\phi} \E _{\mathrm{a}} ^{\bullet, \bullet}, \FF _{\mathrm{a}}
 ^{\bullet, \bullet} ),
 \\
 \label{4.2.2Beintropre}
 \mathrm{Hom} _{\underleftrightarrow{\lim} ^{\mathrm{ab}} \,
 \mathfrak{A}}(\E ^{\bullet, \bullet}, \FF ^{\bullet, \bullet} )=
 \underset{\phi \in M}{\underrightarrow{\lim}}\;
 \mathrm{Hom} _{\mathfrak{A} ^{\Pi} }
 (\widetilde{\phi} \E ^{\bullet, \bullet}, \FF ^{\bullet, \bullet} ).
\end{gather}
We get from (\ref{4.2.2Beintropre}) that the canonical functor
$\underleftrightarrow{\lim} \, \mathfrak{A} \to
\underleftrightarrow{\lim} ^{\mathrm{ab}} \mathfrak{A}$ is fully
faithful.
This enables us to denote by $\underleftrightarrow{\lim} \colon
\mathfrak{A} ^{\Pi} _{\mathrm{a}}\to
\underleftrightarrow{\lim} \, \mathfrak{A}$
and
$\underleftrightarrow{\lim} \colon\mathfrak{A} ^{\Pi}\to 
\underleftrightarrow{\lim} ^{\mathrm{ab}} \mathfrak{A}$
the canonical functors.

\item Let $N (\mathfrak{A})$ be the full subcategory of 
$\mathfrak{A} ^{\Pi} $ whose objects are null in 
$\underleftrightarrow{\lim} ^{\mathrm{ab}} \, \mathfrak{A}$.
Then, the category $N (\mathfrak{A})$ is a Serre subcategory of
$\mathfrak{A} ^{\Pi} $. Moreover, we have the equality $\mathfrak{A}
^{\Pi} / N (\mathfrak{A}) =\underleftrightarrow{\lim} ^{\mathrm{ab}} \,
\mathfrak{A}$. In particular, $\underleftrightarrow{\lim} ^{\mathrm{ab}}
\,\mathfrak{A}$ is an abelian category.
The proof is identical to \cite[1.2.4]{Caind}.

\item Let $\E \in \mathfrak{A}$. For any $c \in \R$, we pose $\E  ^{c}= \E$ if
$c < 0$ and $\E ^{c} = 0 $ otherwise. For any $(a,b) \in \Pi$, we set
$\rho (\E ) ^{a,b}:= \E ^{a} / \E ^{b}$. We get canonically the object $
\rho (\E ) ^{\bullet ,\bullet}\in\mathfrak{A} ^{\Pi} _{\mr{a}}$.
For simplicity, we put $\rho (\E ) := \rho (\E ) ^{\bullet ,\bullet} $
and we get the fully faithful exact functor $\rho \colon \mathfrak{A}
\to \mathfrak{A} ^{\Pi} _{\mr{a}}$.

\item We get similar properties if we replace $S$ by $S ^{f}$.

\end{enumerate}

\subsection{}
\label{0local}
Let $V \hookrightarrow U\hookrightarrow Y \hookrightarrow X$ be open immersions of realizable varieties. The exact
functor $|_{(V,Y)}$ induces the exact functor
\begin{equation*}
 |_{(V,Y)}\colon 
  \underleftrightarrow{\lim} ^{\mathrm{ab}} \,
  F\text{-}\mathrm{Ovhol}(U,X/K)\to 
  \underleftrightarrow{\lim} ^{\mathrm{ab}}\,
  F\text{-}\mathrm{Ovhol}(V,Y/K).
\end{equation*}
Let $ \E ^{\bullet,\bullet} \in \underleftrightarrow{\lim}
^{\mathrm{ab}}\,F\text{-}\mr{Ovhol} (U,X/K)$. 
We remark that
$ \E ^{\bullet,\bullet} =0$ if and only if $ \E ^{\bullet,\bullet} |_{(U,Y)}
=0$. 
Let $\{U_i\}$ be an open covering of $U$. We notice that  $\E
^{\bullet,\bullet} =0$ if and only if $ \E ^{\bullet,\bullet} | (U _i,U
_i) =0$ for any $i$.

\subsection{}
Set 
$\O _{\mathbb{G}  ^{1} _k}:= \O  _{\widehat{\mb{P}} ^{1} _{\V}} (\hdag\{ 0,\infty\} ) _{\Q}$,
$\D _{\mathbb{G}  ^{1} _k}:= \O  _{\widehat{\mb{P}} ^{1} _{\V}} (\hdag\{ 0,\infty\} ) _{\Q} 
\otimes _{\O  _{\widehat{\mb{P}} ^{1} _{\V}} }
\D  _{\widehat{\mb{P}} ^{1} _{\V}}$
and 
$\D ^{\dag} _{\mathbb{G}  ^{1} _k} := \D ^{\dag} _{\widehat{\mb{P}} ^{1} _{\V}} (\hdag\{ 0,\infty\} ) _{\Q}$
and let $t$ be the coordinate of $\widehat{\mb{P}}^1_{\V}$.
We denote by $\O _{\mathbb{G}  ^{1} _k}  [s, s ^{-1}] \cdot t ^{s} $ the free 
$\O _{\mathbb{G}  ^{1} _k}  [s, s ^{-1}]$-module of rank one generated by $t ^{s}$. 
For any integer $a \in \Z$, the free 
$\O _{\mathbb{G}  ^{1} _k}  [s]$-submodule of rank one generated by $s ^{a}t ^{s}$
is denoted by $ s ^{a} \O _{\mathbb{G}  ^{1} _k}  [s]\cdot t ^{s} $ or by 
$\I ^{a} _{\mathbb{G}  ^{1} _k}$.
Following Beilinson's notation, for integers $a\leq b$, we get a free 
$\O _{\mathbb{G}  ^{1} _k}$-module of finite type by putting
\begin{equation*}
 \I ^{a,b} _{\mathbb{G}  ^{1} _k} 
 :=  
\I ^{a} _{\mathbb{G}  ^{1} _k}
 / \I ^{b} _{\mathbb{G}  ^{1} _k}.
\end{equation*}
We define a structure of $\D _{\mathbb{G}  ^{1} _k}$-module on $\O
_{\mathbb{G}  ^{1} _k}  [s, s ^{-1}] \cdot t ^{s} $ so that for $g\in\O
_{\mathbb{G}  ^{1} _k}$ and $l\in\mb{Z}$, we have
\begin{equation}
\label{dfn-connexion}
 \partial_t(s^lg\cdot t^s)=s^l\partial_t(g)\cdot t^s+
  s^{l+1}g/t\cdot t^s.
\end{equation}
Hence, we get a canonical structure of $\D  _{\mathbb{G}  ^{1} _k}
$-module on  $ \I ^{a,b} _{\mathbb{G}  ^{1} _k} $.
We can check that this $\D _{\mathbb{G}  ^{1} _k}$-module structure
on  $ \I ^{a,b} _{\mathbb{G}  ^{1} _k} $ extends to a $\D ^{\dag
}_{\mathbb{G}  ^{1} _k}$-module structure, with which the module is
coherent, and is associated to an overconvergent isocrystal on
$\mathbb{G}  ^{1} _k$ (see the description of \cite[4.4.5]{Be1}).
Indeed, with the notation of \cite[4.4.5]{Be1},
if $ \I ^{a,b, (m)}$ is the free 
$\widehat{\B} ^{(m)} _{\widehat{\mb{P}} ^{1} _{\V}} (\{ 0,\infty\}) _{\Q}$-module of rank one generated by $t ^{s}$
constructed as above (we replace 
$\O _{\mathbb{G}  ^{1} _k}$  
by 
$\widehat{\B} ^{(m)} _{\widehat{\mb{P}} ^{1} _{\V}} (\{ 0,\infty\}) _{\Q}$),
then we can check that the connection on
$ \I ^{a,b, (m)}$ defined with the same formula as \ref{dfn-connexion} extends to a structure of 
$\widehat{\B} ^{(m)} _{\widehat{\mb{P}} ^{1} _{\V}} (\{ 0,\infty\}) \widehat{\otimes} _{\O _{\widehat{\mb{P}} ^{1} _{\V}}}
\widehat{\D} ^{(m)} _{\widehat{\mb{P}} ^{1} _{\V},\Q}$-module.
Moreover, we have an isomorphism
\begin{equation*}
 \I ^{a,b} _{\mathbb{G}  ^{1} _k}\xrightarrow{\sim}
  F^* \I ^{a,b} _{\mathbb{G}  ^{1} _k}\,;\qquad
  s^lg\cdot t^s\mapsto q^lg\otimes (s^l\cdot t^s)
\end{equation*}
with which $\I ^{a,b} _{\mathbb{G}  ^{1} _k}$ is in $F
\text{-}\mathrm{Isoc} ^{\dag}(\mathbb{G}  ^{1} _k /K)$.
The multiplication by $s^n$ induces the isomorphism in $F
\text{-}\mathrm{Isoc} ^{\dag}(\mathbb{G}  ^{1} _k /K)$:
\begin{equation}
\label{sigma_n}
 \sigma ^n \colon\I ^{a,b}_{\mathbb{G} ^{1} _k}\xrightarrow{\sim}
 \I^{a+n,b+n}_{\mathbb{G}^{1}_k}(-n).
\end{equation}

Moreover, there is a non-degenerate pairing
\begin{equation*}
 \I ^{a,b} _{\mathbb{G}  ^{1} _k}\otimes^\dag_{\mc{O}_{\mb{G}^1}}
  \I ^{-b,-a} _{\mathbb{G}  ^{1} _k}\rightarrow
  \mc{O}_{\mb{G}^1_k}(-1);\qquad
  (x(s),g(s))\mapsto\mr{Res}_{s=0}f(s)\cdot g(-s).
\end{equation*}
We can check easily that this pairing is compatible with Frobenius
structure. By using \cite[Prop 3.12]{Aexpl}, the pairing induces an
isomorphism
\begin{equation}
 \label{isomduallogcomm}
  \DD(\I^{a,b}_{\mb{G}^1_k})\xrightarrow{\sim}
  \I^{-b,-a}_{\mb{G}^1_k}.
\end{equation}
As a variant, we put
$\I ^{a,b} _{\mathbb{G}  ^{1} _{k, \mathrm{log}}} 
:= s ^{a} \O _{\widehat{\A} ^{1} _{\V}}  [s] t ^{s} / s ^{b}\O
_{\widehat{\A} ^{1} _{\V}} [s] t ^{s}$. Then $\I ^{a,b} _{\mathbb{G}
^{1} _{k, \mathrm{log}}} $ is a convergent isocrystal on
the formal log-scheme $(\widehat{\A} ^{1} _{\V}, \{0\})$.

\subsection{}
\label{Key-Lemma}
In the rest of the paper, we will keep the following notation. 
Let $X$ be a realizable variety, $f \in
\Gamma(X,\O_X)$ be a fixed function. Put $Z := f ^{-1} (0)
\overset{i}{\hookrightarrow}X \overset{j}{\hookleftarrow} Y:= X
\setminus Z$. 
Let $f | _Y \colon Y \to \mathbb{G} ^{1} _k$ be the morphism induced by
$f$ and put
\begin{equation*}
 \I ^{a,b} _{f}:= (f | _Y) ^{+} (\I ^{a,b} _{\mathbb{G} ^{1} _k} ) [\dim
  Y -1]\in F \text{-}D^{\mr{b}}_{\mr{ovhol}}(Y/K).
\end{equation*}

For $\E\in\mr{Hol}_F(Y/K)$, put $\E ^{a,b}:= \E \otimes \I ^{a,b}
_{f}[-\dim(Y)]$ (see the notation of $\otimes$ after theorem
\ref{sixfuncforma}). Since the functor $-\otimes \I ^{a,b} _{f}
[-\dim(Y)]$ is exact, we get that $\E^{a,b}\in\mr{Hol}_F(Y/K)$ and then
the object $\E ^{\bullet,\bullet}\in
\underleftrightarrow{\lim}\,\mr{Hol}_F(Y/K)$.

\begin{lem*}
 Let $\E \in\mr{Hol}_F(Y/K)$. The canonical
 morphism of $\underleftrightarrow{\lim} \,\mr{Hol}_F(X/K)$
 \begin{equation*}
   \underleftrightarrow{\lim}\,j _! (\E ^{\bullet, \bullet}) \to
   \underleftrightarrow{\lim}\,j _{+} (\E ^{\bullet, \bullet})
 \end{equation*}
 is an isomorphism.
\end{lem*}
\begin{proof}
 We put $d:=\dim(Y)$. Using the five lemma, we may assume that $\E\in
 F\text{-}\mr{Ovhol}(Y/K)$. The proof is divided into several steps.
 \medskip

 \noindent
 {\bf 0)} By \ref{0local}, it is sufficient to check that the canonical
 homomorphism is an isomorphism over $(Y,X)$.
 By abuse of notation in this proof, we still denote by $\E: = \E |
 (Y,X)$ and write $j \colon (Y, X) \to (X, X)$ instead of $(j,
 \mathrm{id})$.
 \medskip

 \noindent
 {\bf 1)} We prove the lemma under the following hypotheses:
 ``Let $\X$ be a smooth formal $\V$-scheme with local coordinates
 denoted by $t _1,\dots , t _d$ whose special fiber is $X$. For any $i=
 1,\dots ,d$, we put $\ZZ _i = V (t _i)$. We suppose that there exist
 an open immersion $U \hookrightarrow Y$ such that $T:= X \setminus U$
 is a strict normal crossing divisor of $X$ and
 an overconvergent $F$-isocrystal $\G$ on $(U,X)/K$ unipotent along $T$ so
 that $\E = \iota _{!} (\G)$, where $\iota \colon(U, X) \to (Y,X)$ is
 the induced morphism of couples
  (see below in the proof for a concrete description of the notion of unipotence).
 We fix $0\leq r '\leq r \leq d$.
 We suppose that the special fiber of $\mc{T}
 := \cup _{1 \leq n \leq r}\ZZ _n$ (resp. 
 $\ZZ :=  \cup _{1 \leq n' \leq r'}\ZZ _{n'}$)
 is $T$ (resp. $Z$).''

 We check the step 1) by induction on the integer $r$ (in the induction, the scheme $X$
 can vary and so is $f$, $Y$, $\E$, $\X$ etc.).
 Where $r =0$, this is obvious. Suppose $r \geq 1$.
 We denote by $\mc{D}:= \cup _{2 \leq n \leq r}\ZZ _n$.
 Consider the following commutative diagram of $F
 \text{-}\mathrm{Ovhol}(X,X/K) ^{\Pi}$
 \begin{equation}
  \notag
   \xymatrix @R=0,3cm{
   {0} 
   \ar[r] ^-{}
   & 
   {\mathcal{H} ^{\dag 0} _{Z _1} j _! (\E ^{\bullet, \bullet})} 
   \ar[r] ^-{}
   \ar[d] ^-{}
   & 
   {j _! (\E ^{\bullet, \bullet})}
   \ar[r] ^-{}
   \ar[d] ^-{}
   & 
   {(\hdag Z _1) j _! (\E ^{\bullet, \bullet})} 
   \ar[r] ^-{}
   \ar[d] ^-{}
   &
   {\mathcal{H} ^{\dag 1} _{Z _1}  j _! (\E ^{\bullet, \bullet})}
   \ar[d] ^-{}
   \\ 
  {0} 
   \ar[r] ^-{}
   & 
   {\mathcal{H} ^{\dag 0} _{Z _1} j _+ (\E ^{\bullet, \bullet})} 
   \ar[r] ^-{}
   & 
   {j _+ (\E ^{\bullet, \bullet})}
   \ar[r] ^-{}
   & 
   {(\hdag Z _1)  j _+ (\E ^{\bullet, \bullet})} 
   \ar[r] ^-{}
   & 
   {\mathcal{H} ^{\dag 1} _{Z _1} j _+ (\E ^{\bullet, \bullet})} 
   }
 \end{equation}
 whose horizontal sequences are exact. 
 By using the induction hypothesis and \ref{0local}, the homomorphism
 $\underleftrightarrow{\lim} \, (\hdag Z _1)  j _! (\E ^{\bullet,
 \bullet})\to\underleftrightarrow{\lim} \, (\hdag Z _1)  j _+ (\E
 ^{\bullet, \bullet})$ is an isomorphism.
 Since $\R \underline{\Gamma} ^{\dag} _{Z _1} j _+ (\E ^{\bullet, \bullet})=0$,
 then it is sufficient to check that $\underleftrightarrow{\lim}
 \,\mathcal{H} ^{\dag i} _{Z _1} j _! (\E ^{\bullet, \bullet})=0$, for
 any $i=0,1$ by the exactness of $\underleftrightarrow{\lim}$.

 We have a strict normal crossing divisor of $\ZZ _1$ defined by
 $\mc{D} _1 := \bigcup _{i=2} ^{r} \ZZ _1 \cap \ZZ _i$. 
 We put $\U := \X \setminus \mc{T}$, and let 
 $i_1\colon \ZZ_1\hookrightarrow\X$ be the canonical closed immersion.
 Since $\G$ is unipotent, following \cite{kedlaya-semistableI},
 this is equivalent to saying that there exists a convergent isocrystal
 $\FF$ on the log scheme $(\X, M _\mc{T})$, where $M _\mc{T}$ means the
 log structure induced by $\mc{T}$ (we keep the same kind of notation
 below), so that $\G \riso (\hdag T  ) (\FF)$. 
 By abuse of notation, we denote by $f = u t _{1}^{a_1} \cdots  t
 _{r'}^{a_{r'}}\in \O _{\X}$ (with $a_i\in\mb{N}$ and $u \in \O ^{*}
 _{\X}$), a lifting of $f$. We put
 \begin{equation*}
 \I ^{a,b} _{f, \mathrm{log}}:= (f ^{\sharp}) ^{*} (\I ^{a,b}
  _{\mathbb{G} ^{1} _{k,\mathrm{log}}} ),
 \end{equation*}
 where $f ^{\sharp}$ is the composition morphism of formal log-schemes
 $f ^{\sharp}\colon(\X, M _\mc{T})  \to (\X, M _\ZZ) \to ( \widehat{\A}
 ^{1} _{\V} , M _{\{ 0\}})$ where the last morphism is induced by
 $f$. We put $\FF  ^{a,b}:= \FF  \otimes ^\dag _{\O _{\X}}\I ^{a,b} _{f
 ,\mathrm{log}}$, which is a convergent isocrystal $\FF$ on the formal
 log scheme $(\X, M _\mc{T})$ with nilpotent residues.
 We put $\U _1 := \ZZ _1 \setminus \mc{D} _1$, and let
 $\iota _1:=(\star,\mathrm{id},\mathrm{id})\colon (U_1,Z_1,\ZZ_1)
 \to(Z_1, Z _1,\ZZ_1)$ be the canonical morphism of frames. Let
 $N_{1,\FF^{a,b}}$ be the action induced by $t _{1} \partial _{1}$ on $i
 _{1} ^{*} ( \FF  ^{a,b})$. We put
 \begin{equation*}
 \J ^{a,b}:= \iota  ^{ +}\bigl(\I ^{a,b} _{f} | _{(Y, X)}\bigr)=
  s ^{a}\, \O_{\X} (\hdag T) _{\Q} [s]\cdot f ^{s} /
  s ^{b}\, \O _{\X} (\hdag T) _{\Q} [s]\cdot f ^{s}.
 \end{equation*}
We have
 \begin{equation*}
 \E ^{a,b}[d]:= \iota  _{!} (\G) \otimes\bigl(\I ^{a,b} _{f} | _{(Y,
  X)}\bigr)\riso\iota _{!}\bigl(\G \otimes \J ^{a,b}\bigr).
 \end{equation*}
 We put $\G ^{a,b}:= \G \otimes \J ^{a,b}[-d]\in F \text{-}\mathrm{Isoc}
 ^{\dag \dag} (U,X/K)$. Since $(\hdag T) ( \I ^{a,b} _{f, \mathrm{log}})
 \riso\J ^{a,b}$, we have $(\hdag T) ( \FF ^{a,b} )\riso\G ^{a,b}$.
 By \cite[3.4.12]{AC}, we get the isomorphisms
 \begin{equation*}
  \mathcal{H} ^{\dag 1} _{Z _1} (j _!  \iota _{!} (\G ^{a,b}))
   \riso 
    i _{1+} \circ \iota _{1!} \circ   (\hdag D _1)
   \left(\coker N_{1,\FF^{a,b}}\right),
   \quad
   \mathcal{H} ^{\dag 0} _{Z _1} (j _! \iota _! (\G ^{a,b}))\riso 
   i _{1+} \circ  \iota _{1!} \circ (\hdag D _1)
   \left(\ker N_{1,\FF^{a,b}}\right).
 \end{equation*}
 Then, by functoriality, it is sufficient to check that
 $\underleftrightarrow{\lim} \, N_{1,\FF^{a,b}}$
 is an isomorphism.
Since
 $N_{1,\FF^{a,b}}= N_{1,\FF} \otimes \mathrm{id}+ \mathrm{id} \otimes
 N_{1,\I ^{a,b} _{f , \mathrm{log}}}$, and since there exists an integer $n$
 (independent of $a$, $b$) such that $N_{1,\FF}^n=0$, then we reduce to
 checking that 
 $\underleftrightarrow{\lim} \, N_{1,\I ^{a,b} _{f ,\mathrm{log}}}$ 
 is an isomorphism, which is obvious since $N_{1,\I ^{a,b} _{f ,
 \mathrm{log}}}$ is the multiplication by $s$.
 \medskip

 \noindent
 {\bf 2)} Finally, let us reduce the lemma to 1). We proceed by
 induction on $\dim X$. We can suppose that $j$ is dominant. Recalling
 that $Y$ being reduced, there exists a dominant open immersion $U\to Y$
 such that $U$ is smooth and $\G := \iota ^{+} (\E) \in 
 F \text{-}\mathrm{Isoc} ^{\dag \dag} (U,X/K)$, where $\iota \colon (U, X)
 \to (Y,X)$. By the induction hypothesis, we can suppose that $\E =
 \iota _{!} (\G)$. Put $T := X \setminus U$. Then, we can suppose that
 $U, Y, X$ are integral and that $\iota$ is affine. 
 Let  $\alpha \colon \widetilde{X}  \to X$ be a proper surjective
 generically finite and \'{e}tale morphism, such that $\widetilde{X} $
 is smooth and quasi-projective,
 $\widetilde{T} := \alpha ^{-1} (T)$ is a strict normal crossing divisor
 of $\widetilde{X} $. We put $\alpha \colon (\widetilde{X},
 \widetilde{X})  \to (X,X)$ (by abuse of notation), $\widetilde{Y} :=
 \alpha ^{-1} (Y) $, $\widetilde{U} := \alpha ^{-1} (U)$, $\beta \colon
 (\widetilde{Y}, \widetilde{X})  \to (Y,X)$,
 $\gamma  \colon (\widetilde{U}, \widetilde{X})  \to (U,X)$,
 $\widetilde{\iota}\colon (\widetilde{U}, \widetilde{X})
 \hookrightarrow (\widetilde{Y}, \widetilde{X})$,
 $\widetilde{j}\colon  (\widetilde{Y}, \widetilde{X})  \hookrightarrow
 (\widetilde{X}, \widetilde{X})$,
 $\widetilde{\G} := \gamma ^{!} (\G)$, 
 $\widetilde{\E}:= \widetilde{\iota} _{!} (\widetilde{\G} )$.
 By Kedlaya's
 semi-stable theorem \cite{Ke}, there exists such a morphism $\alpha$
 satisfying moreover the following property: the object
 $\widetilde{\G}\in F \text{-}\mathrm{Isoc} ^{\dag \dag}
 (\widetilde{U},\widetilde{X}/K)$ is unipotent.
 We know that $\G$ is a direct factor of $\mathscr{H} ^{0} \gamma _{+}
 (\widetilde{\G})$ (see the proof of \cite[6.1.4]{caro_devissge_surcoh}
 at the beginning of p.433).
 Then $\E = \iota _{!} (\G)$ is a direct factor of
 $\iota _{!} \mathscr{H} ^{0} \gamma _{!} (\widetilde{\G})\riso 
 \mathscr{H} ^{0} \beta _{!} \circ \widetilde{\iota} _{!}
 (\widetilde{\G})$.
 Thus we are reduced to checking the lemma for $ \mathscr{H} ^{0} \beta
 _{!} \circ \widetilde{\iota} _{!} (\widetilde{\G})$. We have
 \begin{equation}
  \label{beta0wide1}
   \tag{$\star$}
   \mathscr{H} ^{-d} \beta _{!} \circ \widetilde{\iota} _{!}
   (\widetilde{\G})\otimes (\I ^{a,b} _{f} | _{(Y, X)})
   \riso
   \mathscr{H} ^{-d} \beta _{!}\bigl(\widetilde{\iota} _{!}
   (\widetilde{\G})\otimes \beta ^{+}  (\I ^{a,b} _{f} | _{(Y,
   X)})\bigr)=
   \mathscr{H} ^{-d} \beta _{!} ( \widetilde{\E} \otimes \I ^{a,b}
   _{\widetilde{f}} | _{(\widetilde{Y}, \widetilde{X})})
 \end{equation}
 where $\widetilde{f}= f \circ \alpha$ (the equality comes from
 $\widetilde{\iota} _{!} (\widetilde{\G} )= \widetilde{\E}$ and $\beta
 ^{+}  (\I ^{a,b} _{f} | _{(Y, X)})=  \I ^{a,b} _{\widetilde{f}} |
 _{(\widetilde{Y}, \widetilde{X})}$).
 By applying the exact functor $j _{!}$ (resp.\ $j _+$) to the
 composition isomorphism of (\ref{beta0wide1}), we get the first
 isomorphisms of the following ones:
 \begin{gather*}
  j _{!} (\mathscr{H} ^{-d} \beta _{!} \circ \widetilde{\iota} _{!}
  (\widetilde{\G})\otimes (\I ^{a,b} _{f} | _{(Y, X)}) )
  \riso
  j _{!} \circ \mathscr{H} ^{-d} \beta _{!} ( \widetilde{\E} \otimes \I
  ^{a,b} _{\widetilde{f}} | _{(\widetilde{Y}, \widetilde{X})})
  \riso
 \mathscr{H} ^{-d} \alpha _{!}   \circ \widetilde{j} _{!}(
  \widetilde{\E}\otimes \I ^{a,b} _{\widetilde{f}} | _{(\widetilde{Y},
  \widetilde{X})})
  \\
  j _{+} (\mathscr{H} ^{-d} \beta _{!} \circ \widetilde{\iota} _{!}
  (\widetilde{\G})\otimes (\I ^{a,b} _{f} | _{(Y, X)}) )
  \riso
  j _{+} \circ \mathscr{H} ^{-d} \beta _{!} ( \widetilde{\E} \otimes \I
  ^{a,b} _{\widetilde{f}} | _{(\widetilde{Y}, \widetilde{X})})
  \riso
  \mathscr{H} ^{-d} \alpha _{!}   \circ \widetilde{j} _{+}(
  \widetilde{\E} \otimes \I ^{a,b} _{\widetilde{f}} | _{(\widetilde{Y},
  \widetilde{X})}).
 \end{gather*}
 From 1), the canonical morphism
 $\underleftrightarrow{\lim} \, \widetilde{j} _{!}( \widetilde{\E}
 \otimes \I ^{a,b} _{\widetilde{f}} | _{(\widetilde{Y},
 \widetilde{X})}[-d])\to
 \underleftrightarrow{\lim} \,  \widetilde{j} _{+}( \widetilde{\E}
 \otimes \I ^{a,b} _{\widetilde{f}} | _{(\widetilde{Y},
 \widetilde{X})}[-d])$ is an isomorphism.
 Then so is
 $\underleftrightarrow{\lim} \, \mathscr{H} ^{-d} \alpha _{!}
 \circ \widetilde{j} _{!}( \widetilde{\E}
 \otimes \I ^{a,b} _{\widetilde{f}} | _{(\widetilde{Y},
 \widetilde{X})})
 \riso
 \underleftrightarrow{\lim} \,  \mathscr{H} ^{-d} \alpha _{!}   \circ
 \widetilde{j} _{+}( \widetilde{\E} \otimes \I ^{a,b} _{\widetilde{f}} |
 _{(\widetilde{Y}, \widetilde{X})})$.
\end{proof}

\subsection{}
Let $\E \in\mr{Hol} _F(Y/K)$. With the notation of \ref{Key-Lemma}, we put
$\E _k ^{a,b} := \E ^{\max\{a,k\}, \max\{b,k\}}$ for any integer $k\in
\Z$.
We get $ \E ^{\bullet,\bullet} _{k} \in
\underleftrightarrow{\lim}\,\mr{Hol} _F(Y/K)$.
Now, for $\E\in\mr{Hol} _F(Y/K)$, we put
\begin{equation*}
 \Pi^{a,b}_{!+}( \E) :=
  \underleftrightarrow{\lim}\, j _+ (\E _a ^{\bullet,\bullet}) /
  \underleftrightarrow{\lim}\, j _! (\E _b ^{\bullet,\bullet})
\end{equation*}
in $\underleftrightarrow{\lim}\,\mr{Hol} _F (X/K)$. By Lemma
\ref{Key-Lemma}, this is in fact\footnote{
Since this deduction is formal and not explained in \cite{Be}, further
explanations might be needless for experts, but we point out that the
details are written down in Lichtenstein's thesis \cite[Prop 3.21]{Li}.
However, there is a small mistake in Lichtenstein's argument, as well as
some obvious typos: he claims that there exists an isomorphism
$\widetilde{\varphi}\mc{F}_!^{a,b}
\xrightarrow{\sim}\widetilde{\varphi}\mc{F}_*^{a,b}$ for some
$\varphi\geq\mathbf{1}_{\mb{Z}}$ using the notation in {\it ibid.}, but
this is wrong in general.
This issue can be resolved as follows: We may take $\varphi$ such
that $\widetilde{\varphi}\mc{F}_*^{a,b}\rightarrow
\widetilde{\varphi}\mc{F}_!^{a,b}$ which induces the inverse of $\alpha$
if we pass to the pro-ind category. Now, we consider the last big
diagram in the proof of {\it ibid.}. Because of the mistake, we do not
have the isomorphism $\#$, but we do have a homomorphism
$\#'\colon\underleftrightarrow{\lim}\mc{F}^{a,b}_{*,\varphi\ell}
\rightarrow\underleftrightarrow{\lim}\mc{F}^{a,b}_{!,\varphi\ell}$
making the diagram commutative. Other homomorphisms or isomorphisms
remain to be the same: since the isomorphism $\#$ is used only to show
the existence of the isomorphism
$\mr{coker}_{k,\ell}=\mr{coker}(\star)$, the existence of $\#'$ is
enough to show the equality.
}
in $\mr{Hol} _F(X/K)$, which yields a functor
$\Pi^{a,b}_{!+}\colon\mr{Hol} _F (Y/K)\to\mr{Hol} _F (X/K)$. The
following properties can be checked easily:
\begin{enumerate}
 \item By (\ref{isomduallogcomm}), we have
       $\mb{D}\circ\Pi^{a,b}_{!+}\cong(\Pi^{-b,-a}_{!+}
       \circ\mb{D})(1)$.

 \item The isomorphism $\sigma ^n$ of \ref{sigma_n} induces an isomorphism
       $\Pi^{a,b}_{!+}\xrightarrow{\sim}\Pi^{a+n,b+n}_{!+}(-n)$.
\end{enumerate}
We put $\Psi _f ^{(i)}:=\Pi^{i,i}_{!+}$,
$\Xi^{(i)}_f:=\Pi^{i,i+1}_{!+}$, and put $\Psi _f :=\Psi _f ^{(0)}$,
$\Xi _f:=\Xi _f ^{(0)}$. The isomorphisms
\begin{equation*}
 \underleftrightarrow{\lim}\,j _! (\E_i^{\bullet,\bullet})/
  \underleftrightarrow{\lim}\, j _! (\E _{i+1}^{\bullet,\bullet})
  \cong j_!(\E)(i),\qquad
  \underleftrightarrow{\lim}\,j _+ (\E_i^{\bullet,\bullet})/
  \underleftrightarrow{\lim}\, j _+ (\E _{i+1}^{\bullet,\bullet})
  \cong j_+(\E)(i)
\end{equation*}
induce exact sequences
\begin{equation*}
 0\rightarrow j_!(\E)(i)\xrightarrow{\alpha_-}
  \Xi^{(i)}_f(\E)\xrightarrow{\beta_-}
  \Psi^{(i)}_f(\E)\rightarrow0,\quad
 0\rightarrow \Psi_f^{(i+1)}(\E)\xrightarrow{\beta_+}
  \Xi^{(i)}_f(\E)\xrightarrow{\alpha_+}
  j_+(\E)(i)\rightarrow0.
\end{equation*}
We define a functor $\Phi_f\colon\mr{Hol} _F(X/K)\rightarrow\mr{Hol}
_F(Z/K)$ as follows. Let $\E\in \mr{Hol} _F (X/K)$, and put
$\E_Y:=j^+(\E)$. Let $\gamma_-\colon j_!(\E_Y)\rightarrow\E$ and
$\gamma_+\colon\E\rightarrow j_+(\E_Y)$ be the adjunction
homomorphisms. Consider the sequence
\begin{equation*}
 j_!\E_Y\xrightarrow{(\alpha_-,\gamma_-)}
  \Xi_f(\E_Y)\oplus\E
  \xrightarrow{(\alpha_+,-\gamma_+)}
  j_+(\E_Y).
\end{equation*}
The cohomology of this sequence is $\Phi_f(\E)$.

\begin{rem}
In fact we have checked in the Key Lemma \ref{Key-Lemma}
that the canonical morphism 
$ \alpha ^{\bullet,\bullet} \colon j _! (\E ^{\bullet, \bullet}) \to
j _{+} (\E ^{\bullet, \bullet})$
of $\mr{Hol}_F(X/K) ^{\Pi} _{\mr{a}}$
becomes an isomorphism in 
$(S ^f _{\mr{a}} )^{-1}\mr{Hol}_F(X/K) ^{\Pi} _{\mr{a}}$.
We remark that this is equivalent to saying that 
there exist an integer $N $ large enough
and a morphism 
$\beta ^{\bullet,\bullet} \colon 
j _{+} (\E ^{\bullet +N, \bullet +N}))\to 
j _{!} (\E ^{\bullet, \bullet})$ of 
$\mr{Hol}_F(X/K) ^{\Pi} _{\mr{a}}$ so that the morphisms 
$\alpha ^{\bullet,\bullet}\circ \beta ^{\bullet, \bullet}$ and 
$\beta ^{\bullet, \bullet} \circ \alpha ^{\bullet +N, \bullet +N} $ 
are the
canonical morphisms. 
Since the multiplication by $s ^N$ factors through
$j _{+} (\E ^{\bullet, \bullet }) (N) \riso 
j _{+} (\E ^{\bullet +N, \bullet +N}) 
\to 
j _{+} (\E ^{\bullet, \bullet})$, 
we get that $\coker  \alpha ^{a,b}$
and 
$\ker  \alpha ^{a,b}$
are killed by $s ^{N}$.
For any integer $i\geq 0$, this implies that 
the projective system $\coker  (s ^{i}\alpha ^{a,b} (i))$ stabilizes for
 $b$ large enough (with $a$ and $i$ fixed).
We remark that this limit is isomorphic 
to 
$\Pi^{a,a +i} _{!+}$,
which is the analogue of the remark by Beilinson and Bernstein in
\cite[4.2]{Beil-Bern-Jantzen}.
\end{rem}

\begin{prop}
 \label{mainlemthmcons}
 The functors $\Pi^{a,b}_{!+}$ and $\Phi_f$ are exact. When $\E$ is in
 $\mr{Hol} _F(Z/K)$, then $\E\cong\Phi_f(\E)$ canonically.
\end{prop}
\begin{proof}
 The exactness of $\Pi^{a,b}_{!+}$ follows by that of $j_!$ and
 $j_+$. The exactness of $\Phi_f$ follows since $\alpha_-$ is injective
 and $\alpha_+$ is surjective. The last claim follows by definition.
\end{proof}

\begin{rem*}
 Since we do not use in the proof of the main theorem, we do not go into
 the details, but it is straightforward to get an analogue of
 \cite[Prop 3.1]{Beglue}, a gluing theorem of holonomic modules.
\end{rem*}

\noindent
Tomoyuki Abe:\\
Kavli Institute for the Physics and Mathematics of the Universe (WPI)\\
The University of Tokyo\\
5-1-5 Kashiwanoha, Kashiwa, Chiba, 277-8583, Japan \\
e-mail: {\tt tomoyuki.abe@ipmu.jp}

\bigskip\noindent
Daniel Caro:\\
Laboratoire de Math\'{e}matiques Nicolas Oresme (LMNO)\\
Universit\'e de Caen, Campus 2\\
14032 Caen Cedex, France\\
e-mail: {\tt  daniel.caro@unicaen.fr}

\end{document}